\documentclass[letter, 11pt]{article}
\usepackage{verbatim,amsmath,hyperref,amssymb}
\usepackage{amsthm}
\usepackage{graphics,graphicx}%wrapfig,
\usepackage{mathpazo}
\usepackage{color}
\usepackage{tikz}
\usepackage{multirow}
\usepackage{textcomp}
\usepackage{caption}
\usepackage{subcaption}
\usepackage{titlesec}
\usepackage{booktabs}
\definecolor{darkblue}{rgb}{0,0,.5}
\definecolor{redblue}{rgb}{.5,0,.2}
\definecolor{greenblue}{rgb}{0,.4,.3}
\definecolor{purpleblue}{rgb}{.4,0,.5}
\definecolor{darkgreen}{rgb}{0,0.4,0}
\definecolor{NavyBlue}{cmyk}{0.94,0.54,0,0}
\definecolor{JungleGreen}{cmyk}{0.99,0,0.52,0}
\definecolor{lightgray}{rgb}{0.9,0.9,0.9}

\newcounter{enucount}

\newtheorem{theorem}[enucount]{Theorem}

\newtheorem{proposition}[enucount]{Proposition}

\newtheorem{example}{Example}
\newtheorem{assumption}{Assumption}

\let\emptyset\varnothing

\newcommand{\PP}{\mathbb{P}}
\newcommand{\EE}{\mathbb{E}}

\newcommand{\ZZ}{\mathbb{Z}}
\newcommand{\RR}{\mathbb{R}}

\newcommand{\VV}{\textup{vert}}

\newcommand{\conv}{\textup{conv}}

\newcommand{\ones}{\textbf{1}}
\newcommand{\zeros}{\textbf{0}}

\newcommand{\diag}{\textrm{diag}}
\newcommand{\ipstd}{(\textrm{IP}_\textrm{std})}
\newcommand{\ipnew}{(\textrm{IP}_\textrm{new})}

\title{On a class of stochastic programs with exponentially many scenarios}
\author{Gustavo Angulo\\
Pontificia Universidad Cat\'olica de Chile\\
gangulo@ing.puc.cl}
\date{April 1, 2020}

\begin{document}
\maketitle
\begin{abstract}
 We consider a class of stochastic programs whose uncertain data has an exponential number of possible outcomes, where scenarios are affinely parametrized by the vertices of a tractable binary polytope. Under these conditions, we propose a novel formulation that introduces a modest number of additional variables and a class of inequalities that can be efficiently separated. Moreover, when the underlying polytope is the unit hypercube, we present an extended formulation of polynomial size that can be solved directly with off--the--shelf optimization software. We assess the advantages and limitations of our formulation through a computational study.
\end{abstract}

\section{Introduction}

In stochastic programming, it is usually assumed that the random data of the problem follows a known distribution $\PP$. When $\PP$ is either continuous or finite with a large number of atoms, sampling methods can be used to approximate the true problem by a model involving a reasonable number of scenarios. But what happens when $\PP$ is ``easy'' to describe and still involves an enormous number of possible outcomes? A natural question to ask is whether we  can solve the true problem without relying on sampling methods. One such an example are models with exponentially many scenarios. For instance, scenarios might be given by all subsets of fixed cardinality drawn from a finite ground set. Work along this line is, however, rather limited. In \cite{feige2007robust}, robust combinatorial problems of this type are studied for which approximation algorithms are developed. Further results are obtained in \cite{khandekar2013two}. Below we present a model that illustrates a potential application of implicit representation of scenarios in the context of chance--constrained problems (CCP) and discuss some relevant features and 
questions.

Consider an optimization problem where the data defining the constraints is uncertain. Suppose that the objective of the decision--maker is to find a solution that is feasible with high probability over the possible realization of the data, while incurring the least possible cost. This problem can be formulated as a CCP of the form
\begin{eqnarray*}
& \min& c^\top x \\
& \textrm{s.t.}&\PP(A(\widetilde\omega)x\geq b(\widetilde\omega))\geq 1-\epsilon\\
&& x\in X,
\end{eqnarray*}
where $X\subseteq\RR^n$ imposes deterministic constraints on the decision vector $x$, $c\in\RR^n$ defines a linear objective function, $(A(\widetilde\omega),b(\widetilde\omega))$ denotes the uncertain data, and $\epsilon>0$ is the failure tolerance.

CCPs with independent constraints were introduced in \cite{charnes1959chance}, while jointly--constrained problems were considered in \cite{miller1965chance}. 
%Applications include \cite{lejeune2007efficient,henrion2001,henrion2003optimization}. 
Although CCPs appear in a number of applications, unfortunately they are hard to solve in general due to the nonconvexity of the feasible set.
%\cite{Prekopa2003267}. 
Sampling methods that provide safe approximations and statistical guarantees include \cite{calafiore2005uncertain,calafiore2006scenario,nemirovski2006convex,nemirovski2006probabilistic,luedtke2008sample}.

%Consider an LP where the data involved in the system $Ax\geq b$ is uncertain following a known distribution $\PP$. 

We restrict our attention to the case where the linear system has a single constraint having a finite set of possible realizations parametrized by a left--hand--side vector $a\in\mathcal{A}\subseteq\RR^n$. We are asked for a solution vector $x$ that satisfies all the constraints given by $\mathcal{A}$ with the exception of at most $k-1$ constraints. That is, we want to solve
\begin{eqnarray*}
(\textrm{P}_{k-1})& \min& c^\top x \\
& \textrm{s.t.}&a^\top x\geq b \quad \textrm{for all but at most } k-1\ \textrm{vectors } a\in\mathcal{A}\\
&& x\in X.
\end{eqnarray*}
%where $X\subseteq\{0,1\}^n$ and $\mathcal{A}\subseteq\RR^n$ is a finite set of possible realizations of vector $a$. 
%Suppose that satisfying all the constraints given by $\mathcal{A}$ is too costly or even infeasible. Instead, we are given the chance of discarding $k$ constraints while satisfying the rest. If we think of each vector $a$ being a scenario, then we have a chance-constrained problem where the decision vector $x$ must be feasible in all but $k$ scenarios.
This particular model is known as $(k-1)$--violation linear programming (LP) problem and was introduced in \cite{Roos1994109}. Note that $(\textrm{P}_{k-1})$ can be seen as a CCP with $\PP$ being the uniform distribution over $\mathcal{A}$ and $k-1=\lfloor \epsilon|\mathcal{A}|\rfloor$. 
%Let $\mathcal S$ be the family of $k$-sets of elements of $\mathcal{A}$. We are interested in solving
% \begin{eqnarray*}
% & \min& c^\top x \\
% & \textrm{s.t.}&a^\top x\leq b \quad \forall a\in\mathcal{A}\setminus S\\
% && x\in X\\
% && S\in\mathcal S,
% \end{eqnarray*}
% that is, we have to decide on the vector $x$ and the discarded set $S$ at the same time. 
A standard integer programming (IP) formulation for this problem includes binary variables $z_a$ to decide whether a particular constraint is satisfied or not. It has the form
\begin{eqnarray*}
\ipstd& \min& c^\top x \\
& \textrm{s.t.}&a^\top x+M_az_a\geq b \quad \forall a\in\mathcal{A}\\
&& \sum_{a\in\mathcal{A}}z_a\leq k-1\\
&& x\in X\\
&& z_a \in \{0,1\} \quad \forall a\in\mathcal{A},
\end{eqnarray*}
where $M_a>0$ is a sufficiently large constant such that the constraint $a^\top x+M_a\geq b$ is satisfied for any $x\in X$. In \cite{qiu2014covering} several enhancements over a commercial solver are presented for covering problems of this type.

Now, suppose that $\mathcal{A}$ is extremely large, say exponential in $n$, and thus listing all the variables and constraints in $\ipstd$ is impractical.

\begin{example}\label{ccp_example}
Let $a\in\{0,1\}^n$ represent the set of resources that are available in a given scenario. If at least $p$ of them are available in any scenario, then $\mathcal{A}=\{a\in\{0,1\}^n:\ \sum_{i=1}^n a_i\geq p\}$. Thus $|\mathcal{A}|=\sum_{j=p}^n {n \choose j}$ and in particular $|\mathcal{A}|=2^{n-1}$ for $p=n/2$.
\end{example}

In general, the following issues arise:
\begin{itemize}
 \item The number of variables and constraints in $\ipstd$ is extremely large. Therefore, the natural IP formulation might turn unmanageable due to its size.
 \item The IP formulation makes use of ``big-$M$'' constraints, which are known to lead to weak LP relaxations. Therefore, solving $\ipstd$ might take a considerable amount of time.
 \item Although decomposition methods such as column--and--row generation techniques \cite{muter2013simultaneous,sadykov2013column} could be used to incrementally add variables and constraints, these methods are incompatible with the enhancements in \cite{qiu2014covering} since in the latter the complete formulation must be available a priori. Therefore, decomposition methods of this type are likely to suffer from weak LP relaxations due to ``big--$M$'' constraints.
\end{itemize}

Hence, it is desirable to have at hand alternative formulations that allow for efficient solutions methods. In this work, we make the following contributions towards this end:
\begin{itemize}
 \item We introduce a novel reformulation for $(\textrm{P}_{k-1})$ which does not require binary variables for all scenarios. Under a parametric condition on $\mathcal{A}$, this formulation introduces a modest number of additional variables. Unlike \cite{feige2007robust} and \cite{khandekar2013two}, our setting is rather general and not tailored to specific classes of combinatorial problems.
 \item Under a tractability condition, we show that the constraints in our formulation are easy to separate. Moreover, when the parametrization is given by the unit hypercube, we show how to construct a compact extended IP formulation.
 \item We assess the advantages and limitations of our formulation through a series of computational experiments.
\end{itemize}

The remainder of the paper is organized as follows. In Section~\ref{approach}, we present our assumptions and the general approach. In Section~\ref{formulations}, we derive new formulations for $(\textrm{P}_{k-1})$, and some extensions are discussed in Section~\ref{extensions}. Our computational experiments are reported in Section~\ref{experiments} and concluding remarks in Section~\ref{conclusions}.

\section{The approach}\label{approach}
%Then the above formulation might not be suitable given the large number of constraints that are explicitly included. 
In many instances of $(\textrm{P}_{k-1})$, the description of the problem, in particular that of $\mathcal{A}$, may be fairly simple and yet lead to a complicated IP formulation. Continuing with Example~\ref{ccp_example}, note that the encoding of $\mathcal{A}$ is rather straightforward, we just need $n$ and $p$, but the IP formulation is very large. Given these observations, we propose to use an implicit formulation to solve problems of this type. To that end, let $\mathcal S$ be the family of $(k-1)$--sets of elements in $\mathcal{A}$. We are interested in solving
\begin{eqnarray*}
& \min& c^\top x \\
& \textrm{s.t.}&a^\top x\geq b \quad \forall a\in\mathcal{A}\setminus S\\
&& x\in X\\
&& S\in\mathcal S,
\end{eqnarray*}
that is, we have to decide on the vector $x$ and the discarded set $S$ at the same time without enumerating all the variables and constraints implied by the set of scenarios $\mathcal{A}$.

Let $\phi(x,S):=\min\{x^\top a:\ a\in\mathcal{A}\setminus S\}$. 
Also, for fixed $x$, let $S(x)\in\mathcal S$ denote the set of realizations of $a$ that yield the $k-1$ smallest values of $x^\top a$. With $v_k(x):=\phi(x,S(x))$ being the $k$--th smallest value of $x^\top a$ over $\mathcal{A}$, the problem is equivalent to
\begin{eqnarray*}
(P_v)& \min& c^\top x \\
& \textrm{s.t.}& v_k(x)\geq b\\
&& x\in X.
\end{eqnarray*}
With this formulation, the question becomes whether $v_k(x)$ can be efficiently computed and/or represented or not. Below we make two assumptions on this regard.

\begin{assumption}\label{parametric}
$\mathcal{A}$ is parametrized by an affine function over the vertices of a polytope. More precisely, there exist a polytope $Q\subseteq\RR^m$, a matrix $A\in\RR^{n\times m}$ and a vector $\bar a\in\RR^n$ such that, for any $a\in\mathcal{A}$, there is a unique $z\in\VV(Q)$ satisfying $a=\bar a+Az$, where $\VV(Q)\subseteq Q$ is the set of vertices of $Q$.
\end{assumption}

Assumption~\ref{parametric} let us encode a large number of scenarios implicitly via an affine transformation over a polytope.

\begin{example}\label{ccp_example2}
Suppose we have $n$ independent assets which can yield gains ($g_j>0$) or losses ($l_j<0$). If $\mathcal A$ is the set of uncertain return vectors $a\in\RR^n$, we have
$$\mathcal{A}=\left\{a\in\RR^n|\ a_j\in\{l_j,g_j\}\ \forall i=1,\ldots,n\right\}$$
$$=\left\{a\in\RR^n|\ \exists z\in\{0,1\}^n:\ a_j=l_j+(g_j-l_j)z_j\ \forall i=1,\ldots,n\right\}$$
$$=\left\{a\in\RR^n|\ \exists z\in\VV([0,1]^n):\ a=l+\diag(g-l)z\right\}.$$
Taking $m=n$, $Q=[0,1]^m$, $\bar a=l$ and $A=\diag(g-l)$, we obtain an affine parametrization of $\mathcal{A}$ as in Assumption~\ref{parametric}.
\end{example}

Given $x\in X$, in order to check if $v_{k}(x)\geq b$ holds, we have to compute the $k$--th smallest value of $x^\top a$ over $a\in\mathcal{A}$, or equivalently, the $k$--th smallest value of $x^\top (\bar a+Az)$ over $z\in\VV(Q)$.  This problem reduces to computing the $k$ best basic solutions to an LP, which is NP--hard in general \cite{angulo2014forbidden}. Therefore, for tractability, we restrict our attention to polytopes for which computing the $k$ best basic solutions with respect to a given linear objective function can be done efficiently. This is the case, for instance, of 0/1 polytopes associated to polynomially--solvable combinatorial problems \cite{Lawler,angulo2014forbidden}.

\begin{assumption}\label{tractable}
$Q$ is a tractable binary polytope. More precisely, $\VV(Q)\subseteq\{0,1\}^m$ and both separation and linear optimization over $Q$ can be done in polynomial time with respect to $m$.
\end{assumption}

Under Assumption~\ref{tractable}, $v_k(x)$ is easy to compute. Unfortunately, this function is nonconvex and nonconcave in general.

% \begin{qst}
% Is the concave upper envelope of $v_k(x)$ a tractable or intractable set?
% \end{qst}

Now, let $\sigma_l(x)$ denote the sum of the $l$ smallest values of $x^\top a$ over $a\in\mathcal{A}$, which is a concave function of $x$. We have $v_k(x)=\sigma_k(x)-\sigma_{k-1}(x)$, obtaining
%and let $\bar\sigma_k(x):=\frac{1}{k}\sigma_k(x)$ be their average. 
\begin{eqnarray*}
& \min& c^\top x \\
& \textrm{s.t.}& \sigma_{k}(x)-\sigma_{k-1}(x)\geq b\\
&& x\in X.
\end{eqnarray*}

Let $\bar\sigma_k(x):=\frac{1}{k}\sigma_k(x)$ and note that $\bar\sigma_k(x)\leq\ v_k(x)$ holds. Therefore, $\bar\sigma_k(x)\geq b$ is a convex constraint and the more restrictive problem
\begin{eqnarray*}
(P_{\bar\sigma})& \min& c^\top x \\
& \textrm{s.t.}& \bar\sigma_k(x)\geq b\\
&& x\in X
\end{eqnarray*}
%is a convex problem whose optimal value is 
provides an upper bound for $(P_v)$. Moreover, observe that having constraints on $v_k(x)$ and $\bar\sigma_k(x)$ is analogous to having constraints on the Value--at--Risk (VaR) and Conditional Value--at--Risk (CVaR) of a random variable, respectively. In fact, the CVaR at level $\alpha\in(0,1)$ is related to the average of a random variable over the $\alpha$ upper tail of its distribution, which is exactly the same idea behind $\bar\sigma_k(x)$ with $k$ playing the role of $\alpha$.

\section{Polynomial and exponential formulations}\label{formulations}

In this section, we address the representation of $\sigma_k(x)$ when $Q$ is a binary polytope, 
%begining with the unit hypercube and then the general case, 
which we latter use to represent $v_k(x)$ and reformulate $(\textrm{P}_{k-1})$.

\subsection{Representing $\sigma_k(x)$ with the unit hypercube}

Let us assume that $Q=[0,1]^m$. In this case, computing $\bar\sigma_k(x)$ amounts to finding the average of the $k$ best solutions to $\min\{x^\top (\bar a+Az):\ z\in\{0,1\}^m\}$, which can be formulated as
\begin{eqnarray*}
 \bar\sigma_k(x)=&\min&\frac{1}{k}\sum_{i=1}^k x^\top Az_i+x^\top\bar a\\
 &\textrm{s.t.}&z_i\neq z_j\quad i\neq j\\
 &&z_i\in\{0,1\}^m \quad 1\leq i\leq k.
\end{eqnarray*}
Although linear optimization over binary all--different polytopes can be done in polynomial time \cite{angulo2014forbidden}, no compact linear formulation in $k$ and $m$ is known. On the other hand, given that the objective $x^\top A$ is common to the $k$ vectors $z_i$, we also have
\begin{eqnarray*}
 \bar\sigma_k(x)=&\min&\frac{1}{k}x^\top A y+x^\top\bar a\\
 &\textrm{s.t.}&\sum_{i=1}^kz_i=y\\
 &&z_i\prec z_j\quad i< j\\
 &&z_i\in\{0,1\}^m \quad 1\leq i\leq k,
\end{eqnarray*}
where $\prec$ denotes the usual lexicographical order.
%In this case, we can provide compact extended formulations. 
In this case, we can represent the constraint $\bar\sigma_k(x)\geq b$ with polynomially--many additional variables and constraints as follows.

Let $L_{k,m}$ be the set of $k\times m$ binary matrices $z=[z_1,\ldots,z_k]^\top$ with rows $z_1^\top,\ldots,z_k^\top\in\{0,1\}^m$ in increasing lexicographical order, and let $Y_{k,m}:=\{y\in\ZZ_+^m:\ y=\sum_{i=1}^k z_i\ \textrm{for some}\ z\in L_{k,m}\}$. 
% We have
% \begin{eqnarray*}
%  k\bar\sigma_k(x)=&\min&x^\top A y+kx^\top\bar a\\
%  &\textrm{s.t.}&y\in Y_{k,m}.
% \end{eqnarray*}
Then the constraint $\bar\sigma_k(x)\geq b$ becomes
\begin{equation}\label{con}x^\top A y+kx^\top\bar a \geq kb\quad \forall y\in Y_{k,m}.\end{equation}

Instead of relying on separation to enforce (\ref{con}), we seek a description of $\conv(Y_{k,m})$ of polynomial size with respect to $k$ and $m$ and then apply LP duality to represent (\ref{con}).

In \cite{kaibel2010branched}, an extended formulation for $\conv(L_{k,m})$ with $\mathcal O(k^3m)$ variables and constraints is derived using the framework of branched polyhedral systems. Another formulation of the same size can be directly derived from a dynamic program that solves linear optimization over $L_{k,m}$ \cite{loos2011}. We present a dynamic program with complexity $\mathcal O(k^2m)$ to optimize a linear function over $Y_{k,m}$, from which we derive an extended formulation for $\conv(Y_{k,m})$ with $\mathcal O(k^2m)$ variables and constraints.

\begin{theorem}\label{ef}
 There exist matrices $E$ and $F$, and a vector $g$ such that the system $Ey+Fw=g,\ w\geq 0$ has $\mathcal O(k^2m)$ variables and constraints, and the projection of its solution set onto the $y$--space is equal to $\conv(Y_{k,m})$.
\end{theorem}

\begin{proof}
Given $d\in\RR^m$, let $\gamma^*:=\min\{d^\top y:\ y \in Y_{k,m}\}=\min\{d^\top\sum_{i=1}^k z_i:\ [z_1,\ldots,z_k]^\top \in L_{k,m}\}$. For $1\leq t\leq m$, let $d_{[t]}:=(d_t,\ldots,d_m)\in\RR^{m-t+1}$, and for $1\leq i\leq k$, let $\gamma^t_i:=\min\{d_{[t]}^\top y:\ y\in Y_{i,m-t+1}\}$. With these definitions, we present a dynamic program to compute $\gamma^*=\gamma^1_k$. In what follows, $\zeros$ and $\ones$ represent vectors of 0's and 1's of appropriate size, respectively.

Let $y\in Y_{k,m}$ be an optimal solution defining $\gamma^1_k$, and let $z\in L_{k,m}$ be such that $y=\sum_{i=1}^k z_i$. 
%Let $y\in Y_{k,m}$ and $z\in L_{k,m}$ be such that $y=\sum_{i=1}^k z_i$. 
Observe that the first column of $z$ is either equal to $\zeros$ or $\ones$, or it splits in a sequence of 0's followed by a sequence of 1's. In the former case, $z=[\zeros\ \bar z]$ or $z=[\ones\ \bar z]$ for some $\bar z\in L_{k,m-1}$, and thus $y=(0,\bar y)$ or $y=(k,\bar y)$ for some $\bar y\in Y_{k,m-1}$. Therefore, if $y$ is of this form, it will have objective value equal to $0+\gamma^2_k$ or $kd_1+\gamma^2_k$, whichever is the least, with $\gamma^2_k$ given by an optimal choice of $\bar y$.
%, and computing $\gamma$ reduces to computing $\gamma^2_k$. 
In the latter case, if $k>1$, the split ensures that the first group of rows $[\zeros\ z']$ are lexicographically smaller than the second group $[\ones\ z'']$, where $z'\in L_{l,m-1}$ and $z''\in L_{k-l,m-1}$ for some $1\leq l<k$, and thus $y=(0,y')+(k-l,y'')=(k-l,y'+y'')$ for some $y'\in Y_{l,m-1}$ and $y''\in Y_{k-l,m-1}$. Therefore, if $y$ is of this form, it will have objective value equal to the least of $0+\gamma^2_l+(k-l)d_1+\gamma^2_{k-l}$ over $1\leq l<k$, with $\gamma^2_l$ and $\gamma^2_{k-l}$ given by optimal choices of $y'$ and $y''$, respectively.
%$z=\left[\begin{array}{cc}\zeros&z'\\ \ones&z''\end{array}\right]$ thus computing $\gamma=\gamma^1_k$ reduces to computing $\gamma^2_l$ and $\gamma^2_{k-l}$ over $1\leq l< k$ if $k>1$. 
This idea applies recursively to $\gamma^t_i$ as long as $1\leq t<m$. For $t=m$, we have $L_{1,1}=\{[0],[1]\}$ and $Y_{1,1}=\{(0),(1)\}$, $L_{2,1}=\{[0\ 1]^\top\}$ and $Y_{2,1}=\{(1)\}$, and $Y_{i,1}=\emptyset=L_{i,1}$ for $i>2$, yielding $\gamma^m_1=\min\{0,d_m\}$, $\gamma^m_2=0+d_m$ and $\gamma^m_i=+\infty$ for $i>2$, respectively.

The above construction leads to the dynamic programming formulation \eqref{dp1}--\eqref{dp4} below to compute $\gamma^*=\gamma^1_k$:
\begin{subequations}
\begin{itemize}
 \item For $1\leq t<m$ and $1<i\leq k$:
 \begin{equation}\label{dp1}\gamma^t_i=\min\left\{0+\gamma^{t+1}_i,\ id_t+\gamma^{t+1}_i,\ \min_{1\leq l< i}\left\{0+\gamma^{t+1}_l+(i-l)d_t+\gamma^{t+1}_{i-l}\right\}\right\}\end{equation}
 \item For $1\leq t<m$ and $i=1$:
 \begin{equation}\gamma^t_1=\min\left\{0+\gamma^{t+1}_1,\ d_t+\gamma^{t+1}_1\right\}\end{equation}
 \item For $t=m$:
 \begin{equation}\gamma^m_1=\min\{0,d_m\}\end{equation}
 \begin{equation}\gamma^m_2=0+d_m\end{equation}
 \begin{equation}\label{dp4}\gamma^m_i=+\infty\quad \textrm{for}\ i>2.\end{equation}
\end{itemize}
\end{subequations}

Following the approach of \cite{martin1990polyhedral}, in order to find an extended formulation for $\conv(Y_{k,m})$, we define a hypergraph $H$ representing the dynamic program \eqref{dp1}--\eqref{dp4}. The state space $N$ is identified with variables $\gamma^t_i$. In order to define the hyperarcs in $H$, we first cast \eqref{dp1}--\eqref{dp4} in LP form, yielding the system \eqref{lp1}--\eqref{lp4} below:

%In LP form, we obtain
\begin{subequations}
\begin{itemize}
 \item For $1\leq t<m$ and $1<i\leq k$:
 \begin{equation}\label{lp1}\gamma^t_i\leq\left\{\begin{array}{ccc}
 \displaystyle \gamma^{t+1}_i && (p^t_i)\\
 \displaystyle id_t+\gamma^{t+1}_i && (q^t_i)\\
 \displaystyle \gamma^{t+1}_l+(i-l)d_t+\gamma^{t+1}_{i-l}&\textrm{for}\ 1\leq l< i & (r^t_{il})\end{array}\right.\end{equation}
 \item For $1\leq t<m$ and $i=1$:
 \begin{equation}\label{lp2}\gamma^t_1\leq\left\{\begin{array}{cc}
 \displaystyle \gamma^{t+1}_1 & (p^t_1)\\
 \displaystyle d_t+\gamma^{t+1}_1 & (q^t_1)\end{array}\right.\end{equation}
 \item For $t=m$:
 \begin{equation}\label{lp3}\gamma^m_1\leq\left\{\begin{array}{cc}
 \displaystyle 0 & (p^m_1)\\
 \displaystyle d_m & (q^m_1) \end{array}\right.\end{equation}
 \begin{equation}\label{lp4}\gamma^m_2\leq d_m \quad (r^m_{21}),\end{equation}
\end{itemize}
\end{subequations}
where the objective is to maximize $\gamma^1_k$. For each linear constraint in \eqref{lp1}--\eqref{lp2}, we include in $H$ an hyperarc of the form $(\Gamma,\gamma^t_i)$, where $\Gamma\subseteq\{\gamma^{t+1}_1,\ldots,\gamma^{t+1}_i\}$ is the set of variables in the right-hand side, along with hyperarcs $(\emptyset,\gamma^m_1)$ and $(\emptyset,\gamma^m_2)$ from the boundary constraints \eqref{lp3}--\eqref{lp4}. In this manner, $H$ is acyclic.

From \eqref{lp1}--\eqref{lp4}, by LP duality, we obtain $p,q,r\geq0$ such that
\begin{subequations}
\begin{itemize}
 \item For $1<i\leq k$:
 \begin{equation} \label{ef1} p^t_i-p^{t-1}_i+q^t_i-q^{t-1}_i+\sum_{1\leq l<i}r^t_{il}-\sum_{i<j\leq k}r^{t-1}_{ji}-\sum_{i<j\leq k}r^{t-1}_{j(j-i)}=0\quad \textrm{for}\ 1<t<m\end{equation}
 \begin{equation} p^1_i+q^1_i+\sum_{1\leq l<i}r^1_{il}=\left\{\begin{array}{rcl}0&\textrm{if}&1<i<k\\1&\textrm{if}&i=k\end{array}\right.\quad \textrm{for}\  t=1\end{equation}
 \begin{equation}-p^{m-1}_i-q^{m-1}_i-\sum_{i<j\leq k}r^{m-1}_{ji}-\sum_{i<j\leq k}r^{m-1}_{j(j-i)}=\left\{\begin{array}{rcl}-r^m_{21}&\textrm{if}&2=i\\0&\textrm{if}&2<i\leq k\end{array}\right.\quad \textrm{for}\ t=m\end{equation}
 \item For $i=1$:
 \begin{equation} p^t_1-p^{t-1}_1+q^t_1-q^{t-1}_1-\sum_{1<j\leq k}r^{t-1}_{j1}-\sum_{1<j\leq k}r^{t-1}_{j(j-1)}=0\quad \textrm{for}\ 1<t\leq m\end{equation}
 \begin{equation} \label{ef5} p^1_1+q^1_1=0\quad \textrm{for}\ t=1.\end{equation}
 %\begin{equation} p^m_1-p^{m-1}_i+q^m_1-q^{m-1}_i-\sum_{1<j\leq k}r^{m-1}_{j1}-\sum_{1<j\leq k}r^{m-1}_{j(j-1)}=0\quad t=m\end{equation}
\end{itemize}
To complete the approach in \cite{martin1990polyhedral}, we define $I=\{1,\ldots,k\}\times\{1,\ldots,m\}$ and tag each state $\gamma^t_i\in N$ with reference set $I[\gamma^t_i]=\{1,\ldots,i\}\times\{t+1,\ldots,m\}\subseteq I$. These sets fulfill the consistency and disjointness conditions of \cite{martin1990polyhedral}, namely, for all $(\Gamma,\gamma^t_i)\in H$, we have $I[\gamma^{t+1}_l]\subseteq I[\gamma^t_i]$ for all $\gamma^t_l\in\Gamma$ and $I[\gamma^{t+1}_l]\cap I[\gamma^{t+1}_{l'}]=\emptyset$ for all distinct $\gamma^{t+1}_l,\gamma^{t+1}_{l'}\in\Gamma$.
These conditions imply that the polyhedron \eqref{ef1}--\eqref{ef5} is integral with binary vertices only.

The objective function of the dual problem is to minimize
$$\sum_{1\leq t <m}\left(\sum_{1\leq i\leq k}id_tq^t_i+\sum_{1\leq l< i\leq k}(i-l)d_tr^t_{il}\right)+d_mq^m_1+d_mr^m_{21}.$$
Grouping terms with respect to each $d_t$, we find that $y$ is given by
\begin{equation} \label{ef6} y_t=\left\{\begin{array}{cc}
\displaystyle \sum_{1\leq i\leq k}iq^t_i+\sum_{1\leq l< i\leq k}(i-l)r^t_{il}&\textrm{for}\ 1\leq t<m\\
\displaystyle q^m_1+r^m_{21} &\textrm{for}\ t=m.
\end{array}\right.\end{equation}
\end{subequations}

Taking $w:=(p,q,r)$ and writing (\ref{ef1})--(\ref{ef6}) as $Ey+Fw=g,\ w\geq 0$, we obtain the result as \eqref{ef6} preserves integrality.\qed
\end{proof}

From Theorem~\ref{ef}, we can assume that we have at hand a system of polynomial size such that
$$k\bar\sigma_k(x)=\min\{x^\top Ay:\ Ey+Fw=g,\ w\geq 0\}+kx^\top\bar a.$$
%where $w=(p,q,s)$. 
By LP duality, we have
$$k\bar\sigma_k(x)=\max\{g^\top\pi:\ E^\top \pi=A^\top x,\ F^\top\pi\leq0\}+kx^\top\bar a.$$
Therefore, $\bar\sigma_k(x)\geq b$ can be compactly represented by the system
\begin{eqnarray*}
&&g^\top\pi+k\bar a^\top x\geq kb\\
&& E^\top \pi=A^\top x\\
&&F^\top\pi\leq0.
\end{eqnarray*}

\subsection{Representing $\sigma_k(x)$ with binary polytopes}

Given a binary polytope $Q\subseteq\RR^m$, we have that $Q\cap\{0,1\}^m$ is precisely the set of vertices of $Q$, and thus
\begin{eqnarray*}
 \bar\sigma_k(x)=&\min&\frac{1}{k}\sum_{i=1}^k x^\top Az_i+x^\top\bar a\\
 &\textrm{s.t.}&z_i\neq z_j\quad i\neq j\\
 &&z_i\in Q\cap\{0,1\}^m \quad 1\leq i\leq k.
\end{eqnarray*}

Let $L_{k,Q}:=L_{k,m}\cap Q^k$ and $Y_{k,Q}:=\left\{y\in\ZZ_+^m:\ y=\sum_{i=1}^k z_i\ \textrm{for some}\ [z_1,\ldots,z_k]^\top\in L_{k,Q}\right\}$. Then, along the same lines of the case of the unit hypercube, $(P_{\bar\sigma})$ can be cast as
% \begin{eqnarray*}
% (P_{\bar\sigma})& \min& c^\top x \\
% & \textrm{s.t.}& \left(\sum_{i=1}^k Az_i\right)^\top x+k\bar a^\top x\geq kb \quad \forall [z_1,\ldots,z_k]^\top\in L_{k,Q}\\
% && x\in X,
% \end{eqnarray*}
%where the separation reduces to a minimization problem to finding $k$ best different solutions as above.
\begin{eqnarray}
\nonumber & \min& c^\top x \\
%\nonumber (P_{\bar\sigma})& \min& c^\top x \\
\label{conQ} & \textrm{s.t.}& x^\top A y+k\bar a^\top x\geq kb \quad \forall y\in Y_{k,Q}\\
\nonumber && x\in X.
\end{eqnarray}
Althought it might not be possible to represent (\ref{conQ}) in a compact manner, these inequalities can be efficiently separated, provided that $Q$ is tractable.
\begin{proposition}
 Under Assumption~\ref{tractable}, inequalities (\ref{conQ}) can be separated in polynomial time.
\end{proposition}
\begin{proof}
 To separate (\ref{conQ}) given $x\in X$, we must solve $\min\left\{x^\top A y:\ y\in Y_{k,Q}\right\}$, which is equivalent to $\min\left\{x^\top A \sum_{i=1}^k z_i:\ [z_1,\ldots,z_k]^\top \in L_{k,Q}\right\}$. The latter reduces to computing the $k$ best basic solutions to $\min\left\{x^\top A z:\ z\in Q\right\}$, which can be done in polynomial time \cite{Lawler,angulo2014forbidden}.\qed
\end{proof}

% Returning to $(P_\sigma)$, let $\bar\sigma_+(x)$ denote the average of the $k+1$ smallest values of $x^\top a$ over $\mathcal{A}$. Then $\sigma(x)=(k+1)\bar\sigma_+(x)-k\bar\sigma(x)$ and thus
% \begin{eqnarray*}
% (P_\sigma)& \min& c^\top x \\
% & \textrm{s.t.}& (k+1)\bar\sigma_+(x)\geq b + k\bar\sigma(x)\\
% && x\in X.
% \end{eqnarray*}
% 
% As before, $\bar\sigma_+(x)$ can be lower bounded using linear programming duality. However, to upper bound the concave function $\bar\sigma(x)$ we need to introduce binary vectors $y_i$ representing the $k$ worst realizations of $a$ and linearize the inner product $x^\top A\sum_{i=1}^ky_i$. Still, we obtain an IP formulation for $(P_\sigma)$ of moderate size.

\subsection{Representing $v_k(x)$ and completing the formulation}

In order to represent the constraint $v_k(x)\geq b$, we write it as $\sigma_k(x)-b\geq \sigma_{k-1}(x)$ and consider any feasible solution to the problem defining the right-hand side. Since $\sigma_{k-1}(x)=\min\{\sum_{a\in S}a^\top x:\ S\in\mathcal S\}$, we have $\sigma_{k-1}(x)\leq x^\top A\sum_{i=1}^{k-1} z_i+(k-1)\bar a^\top x$ for any $ z\in L_{k-1,m}$, or equivalently, $\sigma_{k-1}(x)\leq x^\top A y+(k-1)\bar a^\top x$ for any $ y\in Y_{k-1,m}$. We thus arrive at the system
 \begin{eqnarray*}
&&\sigma_{k}(x)-b\geq x^\top A y + (k-1)\bar a^\top x\\
&& y\in Y_{k-1,m},
\end{eqnarray*}
which further leads to
 \begin{eqnarray*}
&&\min\left\{x^\top Aw:\ w\in Y_{k,m}\right\}+k\bar a^\top x-b\geq  x^\top A y + (k-1)\bar a^\top x\\
&& y\in Y_{k-1,m}.
\end{eqnarray*}
%In this case, if $X\substeq\{0,1\}^n$, we have to linearize the inner product $x^\top Ay$. Still, we obtain an IP formulation for $(P_v)$ of moderate size.
In particular, if $Q=[0,1]^m$, from Theorem~\ref{ef} we obtain the equivalent formulation
 \begin{eqnarray*}
\ipnew& \min& c^\top x \\
& \textrm{s.t.}& g^\top\pi+\bar a^\top x\geq b + x^\top A y\\
&& E^\top \pi=A^\top x\\
&&F^\top\pi\leq0\\
&& y\in Y_{k-1,m}\\
&& x\in X.
\end{eqnarray*}
for $(P_v)$. For general binary $Q$, our formulation reads
 \begin{eqnarray*}
& \min& c^\top x \\
& \textrm{s.t.}& w^\top A^\top x+\bar a^\top x\geq b + x^\top A y\quad \forall w\in Y_{k,Q}\\
&&  y\in Y_{k-1,Q}\\
&& x\in X,
\end{eqnarray*}
for which we must resort on separation.

If $X\subseteq\{0,1\}^n$, after linearizing $x^\top Ay$, we obtain an IP formulation for $(P_v)$ without additional binary variables for each scenario in $\mathcal{A}$.

\section{Extensions}\label{extensions}

\subsection{Joint constraints}
Consider the case where we have multiple constraints that must be satisfied at the same time. Whenever one or more of them are violated, we consider it as a failure. The problem to solve is
\begin{eqnarray*}
& \min& c^\top x \\
& \textrm{s.t.}&\left\{a_q^\top x\geq b_q,\ 1\leq q\leq r\right\} \quad \textrm{for all but at most } k-1\ \textrm{matrices}\ [a_1,\ldots,a_r]^\top \in\mathcal{A}\\
&& x\in X,
\end{eqnarray*}
or equivalently,
\begin{eqnarray*}
& \min& c^\top x \\
& \textrm{s.t.}&\min_{1\leq q\leq r}\left\{a_q^\top x-b_q\right\}\geq 0 \quad \textrm{for all but at most } k-1\ \textrm{matrices}\ [a_1,\ldots,a_r]^\top \in\mathcal{A}\\
&& x\in X.
\end{eqnarray*}

Following the approach presented in the previous section, we assume that $a_q=\bar a_q+A_qz$ for $z\in\VV(Q)$. Note that the vector $z$ is the same across all $a_q$, modeling possible dependencies. 
%In order to compute $\min_{z\in Q}\min_{1\leq q\leq r}\left\{x^\top(\bar a_q+A_qz)-b_q\right\}$, we can compute a minimizer $z^q$ of $\min_{z\in Q}\left\{x^\top(\bar a_q+A_qz)-b_q\right\}$ for $1\leq q\leq r$
Now the problem reads
\begin{eqnarray*}
& \min& c^\top x \\
& \textrm{s.t.}&\min_{1\leq q\leq r}\left\{x^\top(\bar a_q+A_qz)-b_q\right\}\geq 0 \quad \textrm{for all but at most } k-1\ \textrm{vectors}\ z\in\VV(Q)\\
&& x\in X.
\end{eqnarray*}
Given $x\in X$, the problem $\min_{z\in Q}\min_{1\leq q\leq r}\left\{x^\top(\bar a_q+A_qz)-b_q\right\}$ has an optimal solution at a vertex of $Q$ since $\varphi(x,z):=\min_{1\leq q\leq r}\left\{x^\top(\bar a_q+A_qz)-b_q\right\}$ is a concave function of $z$. Let $v_k(x)$ be the $k$--th smallest value of $\varphi(x,z)$ over $z\in\VV(Q)$, $\sigma_k(x)$ the sum of the values of its $k$ best basic solutions, and $\bar\sigma_k(x)$ their average.

Assuming $Q\subseteq\RR^m$ is a binary polytope, and recalling $L_{k,Q}=L_{k,m}\cap Q^k$, we have
\begin{eqnarray*}
 \bar\sigma_k(x)=&\min&\frac{1}{k}\sum_{i=1}^k \min_{1\leq q\leq r}\left\{x^\top A_qz_i+x^\top\bar a_q-b_q\right\}\\
 &\textrm{s.t.}&z=[z_1,\ldots,z_k]^\top\in L_{k,Q}.
\end{eqnarray*}
Therefore, we have that $\bar\sigma_k(x)\geq 0$ is equivalent to
\begin{equation}
%\nonumber & \min& c^\top x \\
\label{joint}\sum_{i=1}^k \min_{1\leq q\leq r}\left\{x^\top A_qz_i+x^\top\bar a_q-b_q\right\}\geq 0\quad \forall[z_1,\ldots,z_k]^\top\in L_{k,Q}.
%\nonumber && x\in X.
\end{equation}

%Although the above formulation is not compact, separation is easy.
\begin{proposition}\label{jointsep}
 Under Assumption~\ref{tractable}, inequalities (\ref{joint}) can be separated in polynomial time.
\end{proposition}
\begin{proof}
 Let $x\in X$ and define $Z_0:=\emptyset$. For $i=1,\ldots,k$, let $Z_i:=Z_{i-1}\cup\{z^*_i\}\subseteq\{0,1\}^m$, where $z^*_i$ is an optimal solution to 
 $$\min\left\{\min_{1\leq q\leq r}\left\{x^\top A_qz_i+x^\top\bar a_q-b_q\right\}:\ z_i\in \VV(Q)\setminus Z_{i-1}\right\}.$$
 Since linear optimization over $\VV(Q)\setminus Z_{i-1}$ can be done in polynomial time \cite{angulo2014forbidden}, we can find $z^*_i$ efficiently by solving 
 $$\min\left\{x^\top A_qz_i+x^\top\bar a_q-b_q:\ z_i\in \VV(Q)\setminus Z_{i-1}\right\}$$
 for each $1\leq q\leq r$ and taking $z^*_i$ as an optimal solution to the problem with the least objective value.
 %which can be computed in polynomial time by solving the $r$ inner minimization problems and taking an optimal solution to the one with the least objective value.
 Finally, order and relabel the elements in $Z_k$ so that $z^*=[z^*_1,\ldots,z^*_k]^\top\in L_{k,Q}$ and check whether $x$ satisfies (\ref{joint}) for $z^*$ or not.\qed
\end{proof}

%\begin{qst}
% Is it possible to find a compact formulation for this problem, for instance, based on a dynamic program to solve convex optimization over $L_{k,m}$?
%\end{qst}

In this setting, we have that $v_k(x)\geq 0$ is equivalent to $\sigma_k(x)\geq \sigma_{k-1}(x)$, which can be formulated as
\begin{eqnarray*}
&&\sum_{i=1}^k \min_{1\leq q\leq r}\left\{x^\top A_q\hat z_i+x^\top\bar a_q-b_q\right\}\geq \theta\quad \forall \hat z=[\hat z_1,\ldots,\hat z_k]^\top\in L_{k,Q}\\
&&\sum_{i=1}^{k-1} \min_{1\leq q\leq r}\left\{x^\top A_qz_i+x^\top\bar a_q-b_q\right\}\leq\theta\\
&& z=[z_1,\ldots,z_{k-1}]^\top\in L_{k-1,Q}\\
&&\theta\in\RR,
\end{eqnarray*}
where we have introduced variable $\theta\in\RR$ for simplicity.
From Proposition~\ref{jointsep}, the first set of constraints is easy to separate. To represent the second set of constraints, we include additional variables $w\in\{0,1\}^{(k-1)\times r}$ and $\eta\in\RR^{k-1}$, yielding
\begin{eqnarray*}
& \min& c^\top x \\
& \textrm{s.t.}&\sum_{i=1}^k \min_{1\leq q\leq r}\left\{x^\top A_q\hat z_i+x^\top\bar a_q-b_q\right\}\geq \theta\quad \forall \hat z=[\hat z_1,\ldots,\hat z_k]^\top\in L_{k,Q}\\
&&\sum_{i=1}^{k-1}\eta_i\leq\theta\\
&&x^\top A_qz_i+x^\top\bar a_q-b_q\leq \eta_i+M_qw_{iq}\quad 1\leq i<k,\ 1\leq q\leq r\\
&&\sum_{q=1}^r w_{iq}\leq r-1\quad 1\leq i< k\\
&& x\in X\\
&& z=[z_1,\ldots,z_{k-1}]^\top\in L_{k-1,Q}\\
&&w\in\{0,1\}^{k\times r}\\
&&\eta\in\RR^k\\
&&\theta\in\RR.
\end{eqnarray*}

\subsection{Integral parametrization}
Now, suppose that $\mathcal{A}$ is parametrized by the integral vectors within the rectangle $[0,u]^m$ for some $u\in\ZZ_+$. Again, finding the $k$ best solutions to $\min\{x^\top(\bar a+ Az):\ z\in\{0,\ldots,u\}^m\}$ can be done in polynomial time \cite{angulo2014forbidden} and $\bar\sigma_k(x)$ can be computed as
\begin{eqnarray*}
 \bar\sigma_k(x)=&\min&\frac{1}{k}x^\top A y+x^\top\bar a\\
 &\textrm{s.t.}&\sum_{i=1}^kz_i=y\\
 &&z_i\prec z_j\quad i< j\\
 &&z_i\in\{0,\ldots,u\}^m \quad 1\leq i\leq k.
\end{eqnarray*}
Let $L^u_{k,m}$ be the set of $k\times m$ integral matrices with rows in increasing lexicographical order and having entries in $\{0,\ldots,u\}$. In \cite{angulo2019matrices}, extended formulations for $\conv(L^u_{k,m})$ of size $\mathcal O(k^3m)$, if $u\geq k-1$, and $\mathcal O(uk^3m)$, if $u<k-1$, are obtained via dynamic programming.

Let $Y^u_{k,m}:=\left\{y\in\ZZ_+^m:\ y=\sum_{i=1}^k z_i\ \textrm{for some}\ [z_1,\ldots,z_k]^\top\in L^u_{k,m}\right\}$. Along the lines of Theorem~\ref{ef}, we can build an extended formulation for $\conv(Y^u_{k,m})$ of size $\mathcal O(k^2m)$ or $\mathcal O(uk^2m)$, depending on whether $u\geq k-1$ or not, and then apply LP duality to obtain a compact representation of the constraint $\bar\sigma_k(x)\geq b$.

%Observe that in $(P_\sigma)$ we need to upper bound $\bar\sigma(x)$. However, this time the vectors $y_i$ representing the $k$ worst outcomes of $a$ are integral in $\{0,\ldots,u\}^m$, and therefore linearizing $x^\top A\sum_{i=1}^k y_i$ is more complex than in the binary case.

\section{Computational experiments}\label{experiments}

We compare the standard IP formulation $\ipstd$ and our new formulation $\ipnew$ on problems of the form
\begin{eqnarray*}
&\max&\EE[a]^\top x\\
& \textrm{s.t.}& a^\top x\geq b\ \textrm{for all but at most } k-1\ \textrm{vectors } a\in\mathcal A\\
&&\ones^\top x\leq n/2\\
&& x\in \{0,1\}^n,
\end{eqnarray*}
where $\mathcal A\subseteq\RR^n$ is as in Example 2, that is, we take $Q=[0,1]^n$ and given $l\in\RR^n_-$ and $g\in\RR^n_+$, we set $a=l+\diag(g-l)z$ for $z\in\{0,1\}^n$. For each $n\in\{12,14,16,18\}$, we generate 10 instances of $g$ and $l$, where $g_j$ and $l_j$ are drawn uniformly from $\{1,\ldots,40\}$ and $\{-\left\lceil h_j/2\right\rceil,\ldots,-1\}$, respectively. For each instance, we have $\EE[a]=(g+l)/2\geq \zeros$ and we set $b=-0.2\cdot\ones^\top\EE[a]$. Finally, for each instance, we formulate and solve the problem for $k\in\{10,20,30,40,50\}$ with $\ipstd$ and $\ipnew$.

Our implementation uses IBM CPLEX 12.8 as IP solver. The experiments were run single-threaded on a cluster with 15 nodes equipped with two Intel E5-2470 8-cores processors. In Tables~\ref{n12}--\ref{n18} below we report the CPU time, in seconds, to reach optimality within the default solver tolerance with formulations $\ipstd$ and $\ipnew$.

\begin{table}[h!]
\caption{Time to optimality in seconds, $n=12$.}
%\scriptsize
\begin{center}
\begin{tabular}{crrrrrrrrrr}
%\toprule
	&\multicolumn{2}{c}{$k=10$}&	\multicolumn{2}{c}{$k=20$}	&\multicolumn{2}{c}{$k=30$}	&\multicolumn{2}{c}{$k=40$}	&\multicolumn{2}{c}{$k=50$}\\
\midrule
Instance \#	&	Std	&	New	&	Std	&	New	&	Std	&	New	&	Std	&	New	&	Std	&	New\\
\cmidrule(lr){1-1}		\cmidrule(lr){2-3}	\cmidrule(lr){4-5}	\cmidrule(lr){6-7}	\cmidrule(lr){8-9}	\cmidrule(lr){10-11}
1       & 2.3 & 0.3 & 2.6 & 3.1 & 5.3 & 12.4 & 6.6  & 40.5 & 21.5 & 133.2 \\
2       & 2.2 & 0.3 & 4.0 & 3.0 & 5.8 & 11.2 & 16.5 & 40.5 & 63.5 & 161.0 \\
3       & 1.2 & 0.2 & 2.4 & 1.2 & 2.3 & 7.6  & 3.8  & 26.2 & 5.4  & 78.3  \\
4       & 0.3 & 0.0 & 0.3 & 0.3 & 0.3 & 0.7  & 0.3  & 1.4  & 0.3  & 3.2   \\
5       & 2.3 & 0.3 & 3.0 & 3.6 & 5.2 & 14.7 & 8.7  & 57.6 & 8.0  & 212.1 \\
6       & 1.7 & 0.3 & 2.8 & 1.9 & 4.1 & 5.1  & 5.5  & 18.3 & 7.4  & 67.3  \\
7       & 0.4 & 0.1 & 0.4 & 0.3 & 0.4 & 2.6  & 0.5  & 2.7  & 0.5  & 7.5   \\
8       & 3.1 & 0.4 & 3.9 & 3.5 & 5.6 & 13.6 & 7.0  & 60.9 & 22.9 & 241.0 \\
9       & 2.9 & 0.4 & 3.5 & 3.8 & 4.0 & 15.5 & 6.8  & 62.4 & 8.6  & 249.3 \\
10      & 2.3 & 0.3 & 3.5 & 3.7 & 4.3 & 12.5 & 5.4  & 65.2 & 6.1  & 207.6 \\
\cmidrule(lr){1-1}		\cmidrule(lr){2-3}	\cmidrule(lr){4-5}	\cmidrule(lr){6-7}	\cmidrule(lr){8-9}	\cmidrule(lr){10-11}
Average & 1.9 & 0.3 & 2.6 & 2.4 & 3.7 & 9.6  & 6.1  & 37.6 & 14.4 & 136.1
%\bottomrule
\end{tabular}
\end{center}
\label{n12}
\end{table}

\begin{table}[h!]
\caption{Time to optimality in seconds, $n=14$.}
%\scriptsize
\begin{center}
\begin{tabular}{crrrrrrrrrr}
%\toprule
	&\multicolumn{2}{c}{$k=10$}&	\multicolumn{2}{c}{$k=20$}	&\multicolumn{2}{c}{$k=30$}	&\multicolumn{2}{c}{$k=40$}	&\multicolumn{2}{c}{$k=50$}\\
\midrule
Instance \#	&	Std	&	New	&	Std	&	New	&	Std	&	New	&	Std	&	New	&	Std	&	New\\
\cmidrule(lr){1-1}		\cmidrule(lr){2-3}	\cmidrule(lr){4-5}	\cmidrule(lr){6-7}	\cmidrule(lr){8-9}	\cmidrule(lr){10-11}
1       & 17.2 & 0.4 & 22.9 & 5.1 & 27.5 & 24.1 & 34.6 & 153.3 & 37.5  & 398.9 \\
2       & 7.1  & 0.2 & 6.1  & 1.6 & 5.8  & 7.4  & 7.8  & 11.0  & 8.0   & 46.9  \\
3       & 7.7  & 0.4 & 11.7 & 5.8 & 18.3 & 25.8 & 22.0 & 138.7 & 36.9  & 378.4 \\
4       & 15.0 & 0.4 & 20.6 & 4.6 & 29.1 & 21.6 & 45.8 & 71.7  & 53.0  & 258.3 \\
5       & 10.3 & 0.5 & 17.0 & 6.3 & 21.0 & 26.5 & 36.3 & 119.5 & 57.2  & 425.9 \\
6       & 12.9 & 0.3 & 14.9 & 2.9 & 15.7 & 10.5 & 22.2 & 49.6  & 34.0  & 131.8 \\
7       & 15.1 & 0.3 & 18.7 & 6.9 & 30.9 & 33.9 & 36.2 & 151.7 & 33.1  & 321.4 \\
8       & 12.2 & 0.4 & 16.0 & 5.1 & 19.7 & 18.6 & 25.4 & 102.1 & 31.4  & 253.1 \\
9       & 12.4 & 0.3 & 15.6 & 5.7 & 22.1 & 26.3 & 20.4 & 128.4 & 27.0  & 345.5 \\
10      & 15.2 & 0.5 & 22.6 & 6.8 & 44.7 & 28.9 & 69.2 & 145.5 & 123.9 & 391.4 \\
\cmidrule(lr){1-1}		\cmidrule(lr){2-3}	\cmidrule(lr){4-5}	\cmidrule(lr){6-7}	\cmidrule(lr){8-9}	\cmidrule(lr){10-11}
Average & 12.5 & 0.4 & 16.6 & 5.1 & 23.5 & 22.4 & 32.0 & 107.1 & 44.2  & 295.1
%\bottomrule
\end{tabular}
\end{center}
\label{n14}
\end{table}

\begin{table}[h!]
\caption{Time to optimality in seconds, $n=16$.}
%\scriptsize
\begin{center}
\begin{tabular}{crrrrrrrrrr}
%\toprule
	&\multicolumn{2}{c}{$k=10$}&	\multicolumn{2}{c}{$k=20$}	&\multicolumn{2}{c}{$k=30$}	&\multicolumn{2}{c}{$k=40$}	&\multicolumn{2}{c}{$k=50$}\\
\midrule
Instance \#	&	Std	&	New	&	Std	&	New	&	Std	&	New	&	Std	&	New	&	Std	&	New\\
\cmidrule(lr){1-1}		\cmidrule(lr){2-3}	\cmidrule(lr){4-5}	\cmidrule(lr){6-7}	\cmidrule(lr){8-9}	\cmidrule(lr){10-11}
1       & 94.6  & 0.3 & 121.2 & 3.3  & 146.5 & 12.8 & 189.6 & 77.7  & 240.6 & 252.1 \\
2       & 48.9  & 0.2 & 54.9  & 2.1  & 66.6  & 13.7 & 69.9  & 39.2  & 73.6  & 148.2 \\
3       & 111.6 & 0.7 & 130.4 & 4.3  & 159.7 & 28.8 & 222.0 & 127.6 & 281.1 & 415.5 \\
4       & 48.6  & 0.3 & 62.7  & 2.7  & 63.9  & 8.5  & 78.7  & 43.5  & 89.3  & 102.6 \\
5       & 136.0 & 0.5 & 177.9 & 7.4  & 232.7 & 36.9 & 277.2 & 234.3 & 392.3 & 530.2 \\
6       & 103.7 & 0.5 & 145.5 & 8.3  & 179.7 & 32.3 & 215.4 & 153.1 & 302.6 & 440.5 \\
7       & 87.3  & 0.4 & 127.9 & 7.5  & 168.0 & 28.3 & 322.7 & 122.1 & 277.5 & 524.9 \\
8       & 74.8  & 0.7 & 106.0 & 10.2 & 124.5 & 65.5 & 181.2 & 282.9 & 214.2 & 873.7 \\
9       & 132.2 & 0.9 & 160.8 & 9.6  & 202.1 & 32.9 & 281.2 & 184.1 & 353.5 & 635.3 \\
10      & 133.1 & 0.8 & 187.4 & 8.9  & 242.7 & 38.7 & 327.0 & 260.9 & 444.8 & 841.9 \\
\cmidrule(lr){1-1}		\cmidrule(lr){2-3}	\cmidrule(lr){4-5}	\cmidrule(lr){6-7}	\cmidrule(lr){8-9}	\cmidrule(lr){10-11}
Average & 97.1  & 0.5 & 127.5 & 6.4  & 158.6 & 29.8 & 216.5 & 152.5 & 266.9 & 476.5
%\bottomrule
\end{tabular}
\end{center}
\label{n16}
\end{table}

\begin{table}[h!]
\caption{Time to optimality in seconds, $n=18$.}
%\scriptsize
\begin{center}
\begin{tabular}{crrrrrrrrrr}
%\toprule
	&\multicolumn{2}{c}{$k=10$}&	\multicolumn{2}{c}{$k=20$}	&\multicolumn{2}{c}{$k=30$}	&\multicolumn{2}{c}{$k=40$}	&\multicolumn{2}{c}{$k=50$}\\
\midrule
Instance \#	&	Std	&	New	&	Std	&	New	&	Std	&	New	&	Std	&	New	&	Std	&	New\\
\cmidrule(lr){1-1}		\cmidrule(lr){2-3}	\cmidrule(lr){4-5}	\cmidrule(lr){6-7}	\cmidrule(lr){8-9}	\cmidrule(lr){10-11}
1       & 4319.2 & 0.8 & 4476.0 & 12.3 & 4686.1 & 88.3 & 4830.5 & 411.8 & 5004.5 & 1281.6 \\
2       & 2635.0 & 0.4 & 2794.6 & 3.7  & 2961.4 & 12.9 & 3101.8 & 82.3  & 3371.4 & 166.8  \\
3       & 470.3  & 0.4 & 568.3  & 4.6  & 670.3  & 22.6 & 735.2  & 63.5  & 820.2  & 261.9  \\
4       & 295.0  & 0.5 & 358.8  & 2.3  & 428.4  & 8.6  & 503.6  & 33.3  & 586.4  & 91.3   \\
5       & 2112.9 & 0.7 & 2455.2 & 12.8 & 2754.9 & 71.3 & 2766.4 & 529.5 & 3016.4 & 1705.6 \\
6       & 1702.3 & 0.6 & 1797.3 & 8.5  & 1959.7 & 45.6 & 2052.8 & 340.7 & 2646.3 & 819.5  \\
7       & 220.2  & 0.6 & 294.4  & 8.0  & 1973.2 & 45.4 & 422.7  & 290.5 & 2134.8 & 711.8  \\
8       & 1200.9 & 0.3 & 1323.1 & 4.1  & 1418.3 & 50.8 & 1542.1 & 189.2 & 1614.2 & 383.2  \\
9       & 1585.5 & 0.5 & 1715.4 & 9.6  & 1742.4 & 55.3 & 1989.4 & 531.0 & 2169.8 & 1371.9 \\
10      & 520.7  & 0.6 & 638.2  & 5.8  & 795.3  & 37.3 & 1123.7 & 214.6 & 1574.7 & 565.8  \\
\cmidrule(lr){1-1}		\cmidrule(lr){2-3}	\cmidrule(lr){4-5}	\cmidrule(lr){6-7}	\cmidrule(lr){8-9}	\cmidrule(lr){10-11}
Average & 1506.2 & 0.5 & 1642.1 & 7.2  & 1939.0 & 43.8 & 1906.8 & 268.6 & 2293.9 & 735.9 
%\bottomrule
\end{tabular}
\end{center}
\label{n18}
\end{table}

Recall that $\ipstd$ has $\mathcal O(2^n)$ binary variables and constraints, since $|\mathcal A|=2^n$, while $\ipnew$ has $\mathcal O(n)$ binary variables and $\mathcal O(k^2n)$ continuous variables and constraints. We see that for both formulations, solving times increase with $n$ and $k$, but at different rates. For $\ipstd$, the growth as a function of $k$ is rather mild, while as a function of $n$ it is much more pronounced. For $\ipnew$, the opposite happens, as the increase with $k$ is much steeper than with $n$.

We also observe that for small $k$, $\ipnew$ outperforms $\ipstd$. For $n=12$ and $n=14$, both formulations become comparable for $k=20$ and $k=30$, respectively, while $\ipstd$ outperforms $\ipnew$ for larger $k$. For $n=16$, the figure reverses only for $k=50$, while for $n=18$, $\ipstd$ is not able to compete with $\ipnew$ for any value of $k$.

In light of these results, we conclude that for problems of the form $(\textrm{P}_{k-1})$ with large $|\mathcal A|$ and small $k$, $\ipnew$ might preferable to $\ipstd$.

\section{Concluding remarks}\label{conclusions}
In this work, we present novel formulations for a class of stochastic problems with exponentially many scenarios parametrized by the vertices of a polytope. This approach introduces a modest number of additional variables, and it is shown to outperform a standard IP formulation when the failure tolerance is small. 

We want to highlight the fact that our compact formulations can be readily given to an off--the--shelf solver without the need of an intricate preprocessing or implementation. Formulations that include a large family of inequalities can be tackled with the built--in functions of state--of--the--art solvers. We hope that these characteristics broaden the applicability of the proposed formulations.

A natural question to further address is whether the approach can be extended to non--uniform discrete distributions, to either derive alternative formulations or to provide a fast mechanism to provide strong bounds. Also, objective functions or constraints that consider the expectation operator under similar settings might be considered too.

\section*{Acknowledgments}
The author is grateful to Diego Mor\'an for commenting on an earlier version of this work.

\bibliographystyle{amsplain} 
\bibliography{stochexp}

\end{document}